\documentclass[11pt]{amsart}
\usepackage{hyperref,amsmath,a4wide,graphicx,wasysym,amssymb,mathtools}
\numberwithin{equation}{section}
\usepackage[all]{xy}
\usepackage{tikz}
\usetikzlibrary{arrows,calc,shapes,decorations.pathreplacing,cd}

\newtheorem{lemma}[equation]{Lemma}
\newtheorem{proposition}[equation]{Proposition}

\theoremstyle{definition}
\newtheorem{definition}[equation]{Definition}
\newtheorem{example}[equation]{Example}

\theoremstyle{remark}
\newtheorem*{remark}{Remark}
\newcommand{\mb}[1]{{\mathbf #1}}
\newcommand{\mc}[1]{{\mathcal #1}}

\title{Inessential directed maps and directed homotopy equivalences}
\author{Martin Raussen} 
\address{Department of
  Mathematical Sciences, Aalborg University, Skjernvej 4A,
  DK-9220 Aalborg {\O}st, Denmark} 
\email{raussen@math.aau.dk}
\urladdr{http://people.math.aau.dk/~raussen/} 
\thanks{The author thanks the Hausdorff Research Institute for Mathematics in Bonn, Germany, for its hospitality during two visits as part of the programme Applied and Computational Algebraic Topology in 2017 that allowed him to begin thinking about and discussing the topics dealt with in this paper.\\
He also wishes to thank for a helpful referee report.} 
\begin{document}

\begin{abstract}
A directed space is a topological space $X$ together with a subspace $\vec{P}(X)\subset X^I$ of \emph{directed} paths on $X$. A symmetry of a directed space should therefore respect both the topology of the underlying space and the topology of the associated spaces $\vec{P}(X)_-^+$ of directed paths between a source ($-$) and a target ($+$) - up to homotopy. If it is, moreover, homotopic to the identity map -- in a directed sense -- such a symmetry will be called an inessential d-map, and the paper explores the algebra and topology of inessential d-maps. Comparing two d-spaces $X$ and $Y$ ``up to symmetry'' yields the notion of a directed homotopy equivalence between them. Under appropriate conditions, all directed homotopy equivalences are shown to satisfy a 2-out-of-3 property. Our notion of directed homotopy equivalence does not agree completely with the one defined in \cite{Goubault:17} and \cite{GFS:18}; the deviation is motivated by examples. Nevertheless, directed topological complexity, introduced in \cite{GFS:18} is shown to be invariant under our notion of directed homotopy equivalence. Finally, we show that directed homotopy equivalences result in isomorphisms on the pair component categories of directed spaces introduced in \cite{Raussen:18}.
\end{abstract}

\maketitle

\section{Introduction}
\subsection{Motivation}\label{ss:mot}
Directed Algebraic Topology arose as an attempt to make methodologies from Algebraic Topology useful in the analysis of concurrency phenomena in theoretical computer science; consult eg \cite {Grandis:09, FGHMR:16} for details. It is easy to agree upon the fundamental notions of Directed Algebraic Topology (ie d-space, d-map and d-homotopy, cf Definition \ref{def:d} and \ref{def:d-h}). 

But what are the symmetries of such a d-space, and what is a directed homotopy equivalence; when are two d-spaces homotopy equivalent in the directed sense? The first idea (suggested by several authors) is to adapt directly the notion from ordinary topology: A d-map $f: X\to Y$ would be a directed homotopy equivalence if there exists a d-map $g:Y\to X$ such that $g\circ f$ and $f\circ g$ are d-homotopic (cf Definition \ref{def:d-h}) to the respective identity maps. But this definition does not make sure that the most interesting objects in Directed Algebraic Topology, ie the spaces of d-paths in $X$ from $x$ to $y$ resp.\ of d-paths in $Y$ from $fx$ to $fy$ are related! Have a look at Example \ref{ex:spheres2} for a simple such case. This paper suggests both a notion of symmetry of a d-space (an inessential d-map) and the related notion of directed homotopy equivalence between two d-spaces, and it explores their properties, both algebraically and topologically. More detailed requests to directed homotopy equivalences are formulated in Section \ref{ss:dhe}.

\subsection{Organization of the paper. Results} In Section \ref{s:iness}, we define path space preserving d-maps as those keeping the homotopy types of all path spaces between source and target invariant. Maps from a d-space into itself that are path space preserving and d-homotopic to the identity map are called inessential. It is shown that inessential maps (and also a generalization of those, arising from a certain closure process) are closed under composition and, in some cases, under factorization. 

Section \ref{s:dhe} proposes a new definition of various types of directed homotopy equivalences (and several generalizations). It allows for example to perceive a difference between certain branching spaces and a one point space, in contrast to a similar proposal in \cite{Goubault:17} and \cite{GFS:18}; cf Example \ref{ex:wedge2}. We investigate closure and factorization properties and show that the most important notion of directed homotopy equivalence satisfies the 2-out-of-3 property (cf Section \ref{ss:dhe} and \cite[Definition 1.1.3]{Hovey:99}) in the category of directed spaces; crucial for every thinkable proposal for a model structure on that category.

In Section \ref{s:dTC}, we show that directed topological complexity of d-spaces \cite{GFS:18} is invariant under a sharp version of directed homotopy equivalence. In the final Section \ref{s:pcc}, we show that directed homotopy equivalences result in isomorphisms on the pair component categories of \cite{Raussen:18}. Simple non-trivial examples illustrate the concepts.
\subsection{Elementary notions. Previous results}
Let $I=[0,1]$ denote the unit interval. The following notions are fundamental in Directed Algebraic Topology \cite{Grandis:09} and in applications in concurrency theory \cite{FGHMR:16}:
\subsubsection{d-paths. Traces}
\begin{definition}\label{def:d}
\begin{enumerate}
\item \cite{Grandis:01} A d-space consists of a topological space $X$ together with a subspace $\vec{P}(X)\subset X^I$ that contains the constant paths, is closed under concatenation and under \emph{non-decreasing} reparametrizations $p: I\to I$. Elements of $\vec{P}(X)$ are called d-paths.
\item For $x,y\in X$, we let $\vec{P}(X)_x^y=\{ p\in\vec{P}(X)|\; p(0)=x, p(1)=y\}$ denote the subspace of all d-paths from $x$ to  $y$. We write $x\preceq y\Leftrightarrow\vec{P}(X)_x^y\neq\emptyset$ and call then $(x,y)$ a reachable pair.
\item \cite{Grandis:01} A d-map $f: X\to Y$ between d-spaces $X$ and $Y$ is continuous and satsfies $f(\vec{P}(X))\subseteq\vec{P}(Y)$.
\item The category $\mathbf{dTop}$ has d-spaces as objects and d-maps as morphisms.
\end{enumerate} 
\end{definition}

\begin{definition}\label{def:e}
\begin{enumerate}
\item We let $\vec{X^2}:=\{ (x,y)\in X^2|\; x\preceq y\}$ denote the space of reachable pairs.
\item The end-point map $\vec{e}_X: \vec{P}(X)\to\vec{X^2}$ fibres $\vec{P}(X)$ with non-empty fibres $\vec{P}(X)_x^y$.
\end{enumerate} 
\end{definition}

\begin{remark}
Note that, in contrast to path spaces in ordinary topology, these end point maps very rarely are fibrations: the homotopy types of the fibres do most often vary, d-spaces are in general not homogeneous at all!
\end{remark}

\begin{example}
Simple but fundamental d-spaces are the non-directed interval $I^0$ with $\vec{P}(I^0)=I^I$, the neutral interval $I$ with $\vec{P}(I)$ consisting of the constant paths, and the directed interval $\vec{I}$ with $\vec{P}(\vec{I}):=\{ p\in I^I|\; p \mbox { non-decreasing}\}$. In these cases, the spaces of reachable pairs coincide with $I\times I$, with the diagonal $\Delta (I):=\{ (x,x)|\; x\in I\}$, resp.\ with $\Delta _2(I):=\{ (x,y)\in I\times I|\; x\le y\}$. 
\end{example}

For every d-space $X$, the monoid $\{ p\in\vec{P}(\vec{I})|\; p(0)=0, p(1)=1\}$ of surjective reparametrizations acts on $\vec{P}(X)$ by composition; it identifies d-paths under reparametrization equivalence \cite{FR:07}. Equivalence classes are called \emph{traces}, they are the elements of the quotient trace space $\vec{T}(X)$. The end-point map $\vec{e}_X$ (Definition \ref{def:e}) factors over $\vec{T}(X)$. We will use the same notation for the resulting map $\vec{e}_X: \vec{T}(X)\to\vec{X^2}$ -- with fibres $\vec{T}(X)_x^y$. 
For categorical constructs, traces are better behaved, because concatenation of traces is associative. On the other hand, it is easier to handle, e.g., homotopies of paths than homotopies of traces. 

In applications to concurrency, the d-spaces under consideration are usually directed $\Box$-sets (=pre-cubical sets), or their geometric realizations: 
\begin{definition}\label{df:pcs}
\begin{enumerate}
\item A $\Box$\emph{-set} $X$ (also called a pre-cubical set) is a sequence of disjoint sets $X_n,\; n>0$, equipped with \emph{face maps} $d_i^{\alpha}: X_n\to X_{n-1},\; \alpha\in\{+,-\}, 1\le i\le n$, satisfying the pre-cubical relations: $d_i^{\alpha}d_j^{\beta}=d_{j-1}^{\beta}d_i^{\alpha}$ for $i<j$.\\
Elements of $X_n$ are called $n$-cubes, those of $X_0$ are called vertices.
\item A $\Box$-set $X$ is called \emph{non-self-linked} if every cube $x\in X^n$ has $2^i$${n}\choose{i}$ \emph{different} iterated $(n-i)$-faces, $0<i\le n$. 
\item The \emph{geometric realization} of a pre-cubical set $X$ is the d-space
\[|X|=\bigcup_{n\ge 0}X_n\times \vec{I}^n/_{[d_i^{\alpha}(c),x]\sim [c,\delta_i^{\alpha}(x)]}\] 
with $\delta_i^{\alpha}(x_1,\dots ,x_{n-1})=(x_1,\dots, x_{i-1}, s_{\alpha}, x_i,\dots , x_{n-1})$ and $s_{\alpha}= 0$ (resp.\ $1$) for $\alpha =-$ (resp.\ $\alpha =+$).
\item A path $p\in\vec{P}(|X|)$ is directed if there are $0=t_0<t_1<\dots <t_k=1$, cubes $c_i\in X_{n_i}$ and directed paths $p_i:[t_{i-1},t_i]\to \vec{I}^{n_i}$ with $p(t)=[c_i,p_i(t)]$ for $t\in [t_{i-1},t_i]$.
\end{enumerate}
\end{definition}

We allow ourselves to use the most convenient setting (paths resp.\ traces) in a given situation because of 

\begin{proposition}\cite[Proposition 2.16]{Raussen:09}
For the geometric realization of a $\Box$-set $X$, the quotient maps $\vec{P}(X)_{x_1}^{x_2}\to\vec{T}(X)_{x_1}^{x_2},\; (x_1,x_2)\in\vec{X}^2$, are homotopy equivalences.
\end{proposition}
  
 \noindent \begin{minipage}{0.75\textwidth}
\begin{example}\label{ex:sphere}
\cite{Ziemianski:18} Let $n>1$ and $\partial \Box_n:=\{ \mb{x}=(x_1,\dots, x_n)\in I^n|\; \exists i: x_i=0 \mbox{ or } x_i=1\}$ with the d-structure inherited from the standard product d-structure on $\mb{R}^n$, ie $\vec{P}(\partial \Box_n)=(\partial\Box_n)^I\cap\vec{P}(\mb{R}^n)$. Then the trace space $\vec{T}(\partial \Box_n)_{\mb{x}}^{\mb{y}}$ is contractible unless $\{ i|\; x_i=0\}=\{ i|\; y_i=1\}$.  If these two sets agree, then $\vec{T}(\partial \Box_n)_{\mb{x}}^{\mb{y}}$ is homotopy equivalent to $S^{k-2}$ with $k=|\{ i|\, x_i=0, y_i=1\}|$.

In the figure on the right, the trace space from the bottom to the top vertex is homotopy equivalent to $S^1$; from a point on a bottom edge to a point on a top edge in the same colour homotopy equivalent to $S^0$, and contractible (if not empty) otherwise.
\end{example}
\end{minipage}\hfill 
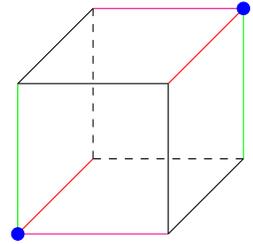
\begin{minipage}{0.2\textwidth}
\begin{tikzpicture}
\draw[magenta] (0,0) -- (2,0);
\draw[magenta] (1,3) -- (3,3);
\draw[green] (0,0) -- (0,2);
\draw[green] (3,1) -- (3,3);
\draw[red] (0,0) -- (1,1);
\draw[red] (2,2) -- (3,3);
\draw (2,0) -- (3,1);
\draw (0,2) -- (1,3);
\draw (0,2) -- (2,2);
\draw (2,0) -- (2,2);
\draw[dashed] (1,1) -- (1,3); 
\draw[dashed] (1,1) -- (3,1);

\draw[blue,fill=blue] (0,0) circle (.5ex);
\draw[blue,fill=blue] (3,3) circle (.5ex);
\end{tikzpicture}
\end{minipage}

\medskip The following results show that trace spaces of $\Box$-sets can be quite complicated:

\begin{proposition}
\begin{enumerate}
\item \cite[Proposition 3.15]{Raussen:09} Path and trace spaces in a non-self-linked $\Box$-set have the homotopy type of a CW-complex.
\item \cite{Ziemianski:16}
For every finite CW complex $Z$, there exists a $\Box$-set $X$ (without directed loops) with vertices $x_1,x_2\in X$ such that $\vec{T}(X)_{x_1}^{x_2}$ is homotopy equivalent to $Z$.
\end{enumerate}
\end{proposition}

A d-map $f: X\to Y$ induces maps $\vec{f^2}: \vec{X^2}\to\vec{Y}^2,\; (x_1,x_2)\mapsto (fx_1,fx_2)$, $\vec{P}(f): \vec{P}(X)\to\vec{P}(Y)$ and $\vec{T}(f): \vec{T}(X)\to\vec{T}(Y), \sigma\mapsto f\circ \sigma$. They fit into a commutative diagram (depicted for $\vec{T}$):
\[\begin{xymatrix}{\vec{T}(X)_{x_1}^{x_2}\ar[r]^{\vec{T}(f)_{x_1}^{x_2}}\ar@{^{(}->}[d] & \vec{T}(Y)_{fx_1}^{fx_2}\ar@{^{(}->}[d]\\
\vec{T}(X)\ar[r]^{\vec{T}(f)}\ar[d]_{\vec{e}_X} & \vec{T}(Y)\ar[d]^{\vec{e}_Y}\\
\vec{X^2}\ar[r]^{\vec{f^2}} & \vec{Y}^2
}\end{xymatrix}\]
with restricted maps $\vec{T}(f)_{x_1}^{x_2}: \vec{T}(X)_{x_1}^{x_2}\to\vec{T}(Y)_{fx_1}^{fx_2}$ on fibres over $(x_1,x_2)\in\vec{X^2}$. Similarly for path spaces.
\subsubsection{Directed homotopies}
\begin{definition}\label{def:d-h}
\begin{enumerate}
 \item A d-map $H:X\times\vec{I}\to X$ is called a \emph{d-homotopy} from $H_0$ to $H_1$. $H_1$ is then called future d-homotopic to $H_0$ and $H_0$ is called past d-homotopic to $H_1$.  Notation: $H_0\nearrow H_1$ or $H_1\nwarrow H_0$.
 \item Two d-maps maps $f, g: X\to Y$ are called \emph{d-homotopic} if there exists a zig-zag sequence $f=f_0\nearrow f_1\nwarrow f_2\nearrow\dots \nwarrow f_n=g$ of d-homotopies.
\item A d-map $H:X\times I\to Y$ is called a \emph{neutral d-homotopy} between $H_0$ and $H_1$. The d-maps $H_0$ and $H_1$ are called \emph{neutrally d-homotopic}.
\item Two d-maps $f,g: X\to Y$ are called \emph{dihomotopic} if there exists a neutral d-homotopy with $H_0=f$ and $H_1=g$.
    \end{enumerate}
\end{definition}

We will be particularly interested in these notions for endo-d-maps $f: X\to X$ from a space into itself that are related to the identity map $id_X$ by a d-homotopy.

\begin{example}\label{ex:wedge}
The two relations future and past d-homotopic differ essentially even in simple examples:
Consider the d-space $B=\{ (x,y)\in\mb{R}^2|\; x,y\ge 0, xy=0\}=$\begin{tikzpicture} \draw [->] (-0.2,-0.2) -- (-0.2,0.1);  \draw [->] (-0.2,-0.2) -- (0.1,-0.2);\end{tikzpicture} with the d-structure inherited from the standard d-structure of $\mb{R}^2$ and modelling a future branching. 

A d-homotopy with $H_0=id_B$ preserves the two branches $x=0$ and $y=0$ and thus their intersection point $O=(0,0)$; in particular, $H_1$ cannot be constant. 

On the other hand, the d-homotopy $H(x,y;t)=(tf_1(x,y),tf_2(x,y))$ connects the constant map $H_0$ with every d-map $f=(f_1,f_2):B\to B$. Hence all d-maps from $B$ to $B$ are dihomotopic to each other.

In greater generality, if $H:X\times\vec{I}\to X$ is a d-homotopy with $H_1(X)=\{x_0\}$ for some $x_0\in X$, then there exists a d-path from $x$ to $x_0$ for every $x\in X$.
\end{example}

The following example shows that d-homotopies do not preserve homotopy types of path spaces, and therefore the naive notion of directed homotopy equivalence mentioned in Section \ref{ss:mot} is not satisfactory:

\begin{example}\label{ex:spheres2}
Let $p: I\to I$ be the piecewise linear reparametrization with $p(0)=0, p(0.5)=p(1)=1$; there is a convex d-homotopy $H$ with $H_0=id_I$ and $H_1=p$. Hence there is also a convex d-homotopy on $\partial\Box_n$ with $H_0=id$ and $H_1=p^n: \partial\Box_n\to \partial\Box_n$. Hence $p$ and $id$ are d-maps that are inverse to each other up to d-homotopy.

\noindent \begin{minipage}{0.8\linewidth}Let $\mb{0}=(0,\dots ,0)$ and $\mb{x}=(1, 0.5,\dots , 0.5)$ and note that $p^n(\mb{0})=\mb{0}$ and $p^n(\mb{x})=\mb{1}$. Then $\vec{T}(\partial \Box_n)_{\mb{0}}^{\mb{x}}$ is contractible whereas $\vec{T}(\partial \Box_n)_{\mb{0}}^{\mb{1}}$ is homotopy equivalent to $S^{n-2}$ (cf Example \ref{ex:sphere}). In particular, $\vec{T}(p): \vec{T}(\partial\Box_n)_{\mb{0}}^{\mb{x}}\to\vec{T}(\partial\Box_n)_{\mb{0}}^{\mb{1}}$ is not a homotopy equivalence.
\end{minipage}\qquad
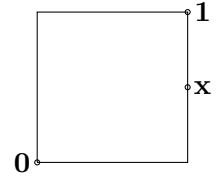
\begin{minipage}{0.2\linewidth}
\begin{tikzpicture}
\draw (0,0) --(0,2) --(2,2) -- (2,0) -- (0,0);
\draw node at (2.2,2) {$\mb{1}$};
\draw node at (2.2,1) {$\mb{x}$};
\draw (2,1) circle (.2ex);
\draw (2,2) circle (.2ex);
\draw (0,0) circle (.2ex);
\draw node at (-0.2,0) {$\mb{0}$};
\end{tikzpicture}
\end{minipage}
\end{example}

\subsection{Requests to the notion of direceted homotopy equivalence}\label{ss:dhe}
We can now formulate reasonable requests that a notion of directed homotopy equivalence should satisfy: A directed homotopy equivalence should 
\begin{enumerate}
\item have a homotopy inverse $g:Y\to X$ such that $g\circ f$ is homotopic to $id_X$ and $f\circ g$ is homotopic to $id_Y$ -- in one of the flavours of directed homotopy from Definition \ref{def:d-h}.
\item satisfy that all maps $\vec{T}(f)_{x_1}^{x_2}: \vec{T}(X)_{x_1}^{x_2}\to\vec{T}(Y)_{fx_1}^{fx_2},\;  (x_1,x_2)\in\vec{X^2}$ are (weak) homotopy equivalences in the classical sense. Moreover,
\item In the category $\mathbf{dTop}$ (objects: d-spaces, morphisms: d-maps; cf fx \cite{FGHMR:16}), the directed homotopy equivalences enjoy the 2-out-of-3 property (cf fx \cite[Definition 1.1.3]{Hovey:99}).
\end{enumerate}
Ir is furthermore desirable to obtain a version of directed homotopy equivalence for which a branching space like  $B=$\begin{tikzpicture} \draw [->] (-0.2,-0.2) -- (-0.2,0.1);  \draw [->] (-0.2,-0.2) -- (0.1,-0.2);\end{tikzpicture} from Example \ref{ex:wedge} is \emph{not} contractible in a directed sense.

I have called item (1) alone in the list above the naive definition of directed homotopy equivalence since it does not take care of path spaces at all. This problem is taken care of in the definition presented in \cite{Goubault:17} and \cite{GFS:18}: it requires (2) in a coherent sense (cf Definition \ref{def:cpsp}) but it does not ensure (1) nor (3).
 
Our definition of directed homotopy equivalence (in Definition \ref{def:dhe}) does not satisfy alle requests on the nose. Item (1) is part of the definition, Lemma \ref{lem:F-eq} shows that (2) is satisfied ``up to inessential endo-maps'' and Proposition \ref{prop:2-3} shows that (3) is satisfied for neutral d-homotopy equivalences. The branching space $B$ from Example \ref{ex:wedge} is not future homotopy equivalent to a one point d-space. Note also the more elaborate Example \ref{ex:wedge2}(2) of a directed graph $W$ which is not homotopy equivalent to the one point space $O$ in any directed sense according to our definition -- although one can easily check that the constant map $c:W\to O$ is a classical homotopy equivalence satisfying (2) above. 

\section{(Rather) Inessential d-maps}\label{s:iness}

\subsection{Path space preserving d-maps and d-homotopies}\label{ss:psp}

Ordinary topological spaces give rise to loop spaces that are interesting to study on their own behalf. d-spaces give rise to, and organise, many spaces of (directed) paths, one path or trace space $\vec{T}(X)_{x_1}^{x_2}$ for every pair of points $(x_1,x_2)\in\vec{X}^2$; cf Definition \ref{def:e}(1). For reasonable d-spaces, like the $\Box$-spaces from Definition \ref{df:pcs}, these path spaces depend only mildly on the pair of points; they are stable within so-called components \cite{Ziemianski:18, Raussen:18}. If we wish that a d-map not only relates the topology of its domain and target, but also the topology of all assembled path spaces, we need to add the following requirement:  

\begin{definition}\label{def:psp}
  \begin{enumerate}
  \item A d-map $f:X\to Y$ is called \emph{path space preserving}
    (psp for short) if
    $\vec{T}(f)_{x_1}^{x_2}: \vec{T}(X)_{x_1}^{x_2}\to\vec{T}(Y)_{fx_1}^{fx_2}$ is a weak homotopy equivalence for all $(x_1,x_2)\in\vec{X}^2$.
    \item A d-homotopy $H: X\times\vec{I}\to Y$, resp.\ a neutral d-homotopy $H:X\times I\to Y$, is called psp if every parameter d-map $H_t:X\to Y,\; t\in I,$ is psp.
    \item Adding psp to the requirements in Definition \ref{def:d-h} leads to the relations \emph{future/past}, resp.\ \emph{neutrally} \emph{psp d-homotopic}.
  \end{enumerate}
  \end{definition}
  
\begin{example}\label{ex:square}
This example investigates psp maps $f:\partial\Box_n\to\partial\Box_n, n\ge 2$, cf  Example \ref{ex:sphere}(1). It shows that the psp-properties in Definition \ref{def:psp} impose severe restrictions. First of all, let $\mb{0}$ and $\mb{1}$ denote the extreme vertices in $\Box_n$. The pair $(\mb{0},\mb{1})$ is the only one giving rise to a trace space $\vec{T}(\partial\Box_n)_{\mb{0}}^{\mb{1}}\simeq S^{n-2}$. As a consequence, every psp endo d-map $f:\partial\Box_2\to\partial\Box_2$ has to fix both $\mb{0}$ and $\mb{1}$ and the complement of the set $\{\mb{0},\mb{1}\}$. 

The case $n=2$ has to be dealt with separately: The identity map and the reflection in the diagonal are (non-d-homotopic) psp maps. Every d-map $f: \partial\Box_2\to\partial\Box_2$ that preserves $\mb{0}, \mb{1}, \{ (x,y)|\; x>y\}$ and $\{ (x,y)|\; x<y\}$ is neutrally psp d-homotopic to the identity map. The vertices $(0,1)$ and $(1,0)$ are not preserved, in general.

In the following, let $n>2$: Let $V=\{ 0,1\}^n$ denote the set of all vertices, let $V_k\subset V, 0\le k\le n$ denote the set of all vertices with exactly $k$ coordinates equal to $1$. We want to establish, that a psp map $f: \partial\Box_n\to\partial\Box_n$ preserves all vertex sets $V_k, k\neq 1,n-1$. To this end, let $B_k$ denote the set of all $\mb{x}$ with exactly $n-k$ coordinates equal to $0$ and none equal to $1$, resp.\ $B^k$ the set of all $\mb{x}$ with exactly $n-k$ coordinates $1$ and none equal to $0$ (unions of interiors of lower, resp\ upper $k$-faces).   By Example \ref{ex:sphere}(1), $f$ preserves $B_k$ and $B^k$ for $k<n-1$. By continuity, $f$ preserves also their closures and intersections of those. In particular, for $k\neq 1, n-1$, $f$ preserves $V_k=\bar{B}_{n-k}\cap\bar{B}^k$. 

On the other hand, every permutation $\sigma\in\Sigma_n$ defines a psp d-map $\bar{\sigma}: \partial\Box_n\to\partial\Box_n,\\ (x_1,\dots ,x_n)\to (x_{\sigma (1)},\dots ,x_{\sigma (n)})$, that permutes elements in $V_k$. 

Now we want to prove that every \emph{future} psp d-homotopy $H: \partial\Box_n\times\vec{I}\to\partial\Box_n$ from the identity map (cf Definition \ref{def:d-h}(1)) has to \emph{fix all vertices} in $V$:
Such a psp d-homotopy $H$ preserves each \emph{lower} $k$-face in $B_k$, and each \emph{upper} $l$-face in $B^{l}$ (characterized by \emph{which} coordinates are $0$, resp.\ $1$) and their closures for $k,l<n-1$.  In particular, every vertex in $V_k$ is fixed for $k\neq 1, n-1$. Now suppose $k=1$. Since every lower 1-face and its closure are preserved by $H$,  its single upper boundary vertex (with $n-1$ coordinates $0$ and 1 coordinate $1$) needs to be fixed since $H$ is a future d-homotopy.

\noindent\begin{minipage}{0.7\textwidth}
Finally, we consider the vertex $\mb{x}=[1,\dots ,1,0]$; more generally a vertex in $V_{n-1}$. Let $f:=H_1$. Suppose $f(\mb{x})=(1,\dots ,1,t)$ with $0<t<1$. Since $f$ is homotopic to $id_X$ and thus surjective, there exists $\mb{y}=(y_1,\dots ,y_{n-1},0)\neq\mb{x}$ with $f(\mb{y})=\mb{x}$; without restriction, we assume $y_{n-1}<1$. The future of $\mb{y}$ (shaded in the adjacent figure) is mapped onto the edge connecting $\mb{x}$ and $\mb{1}$. It contains the element $\mb{z}=(1,\dots ,y_{n-1},1)$ which therefore has to be mapped to $\mb{1}$. Contradiction, since a psp d-map maps only $\mb{1}$ into $\mb{1}$. 
\end{minipage} \hfill
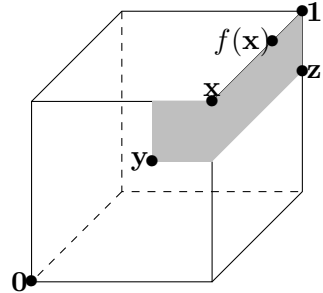
\begin{minipage}{0.3\textwidth}
\begin{tikzpicture}[scale=0.8]
\draw (0,0) -- (3,0) -- (3,3) -- (0,3) --(0,0);
\draw (3,0) -- (4.5,1.5) -- (4.5,4.5) -- (3,3);
\draw (0,3) -- (1.5,4.5) -- (4.5,4.5);
\draw [dashed] (1.5,4.5) -- (1.5,1.5) -- (4.5,1.5);
\draw [dashed] (1.5,1.5) -- (0,0);
\fill[fill=lightgray] (2,3) -- (2,2) -- (3,2) -- (3,3);
\fill[fill=lightgray] (3,3) -- ((3,2) -- (4.5,3.5) -- (4.5,4.5) -- (3,3);
\node at (3,3.2) {$\mb{x}$};
\node at (3.5,4) {$f(\mb{x})$};
\node at (1.8,2) {$\mb{y}$};
\node at (4.7,3.5) {$\mb{z}$};
\node at (4.7,4.5) {$\mb{1}$};
\node at (-0.2,0) {$\mb{0}$};
\node at (4.5,4.5) {$\bullet$};
\node at (4.5,3.5) {$\bullet$};
\node at (3,3) {$\bullet$};
\node at (2,2) {$\bullet$};
\node at (4,4) {$\bullet$};
\node at (0,0) {$\bullet$};
\end{tikzpicture}
\end{minipage}

We conclude that $f$ preserves all faces of $\partial\Box_n$ (including their boundaries). In fact, $f$ preserves also the \emph{interiors} of faces: Assume  $f(x_1,\dots ,x_n)=(x_1',\dots ,x_{n-1}',1)$ with $x_n<1$. Since $f$ is a d-map and since $n>2$, $f(1,\dots ,1,x_n)=\mb{1}$. Contradiction!

A similar proof works for maps that are past psp d-homotopic to the identity.

\end{example}

\subsection{Coherently path space preserving maps}
Definition \ref{def:psp} relates trace spaces in the domain $X$ and the codomain $Y$ to each other, but these equivalences may not be coherent: their ``inverses'' -- for varying pairs $(x_1,x_2)\in\vec{X}^2$ are not necessarily collected into a  continuous map. We need to formulate an additional requirement in order to achieve this:

For a d-map $f: X\to Y$, consider the commmutative diagram
\begin{equation}\label{eq:coh}
\xymatrix{\vec{T}(X)\ar[d]_{e_X}\ar[r]^{\vec{T}(f)} &  (\vec{f^2})^*\vec{T}(Y)\ar[d]^{(\vec{f^2})^*(e_Y)}\ar[r] & \vec{T}(Y)\ar[d]^{e_Y}\\
\vec{X}^2\ar[r]^{=} & \vec{X}^2\ar[r]^{\vec{f^2}} & \vec{Y}^2}
\end{equation}
with $(\vec{f^2})^*(e_Y)$ the pullback of $e_Y$ along with $\vec{f^2}$; ie $(\vec{f^2})^*\vec{T}(Y)$ has fibre $\vec{T}Y_{fx_1}^{fx_2}$ over $(x_1,x_2)\in\vec{X}^2$.

\begin{definition}\label{def:cpsp}
\begin{enumerate}
\item A psp d-map $f: X\to Y$ is called \emph{coherently} \emph{psp} if the map $\vec{T}(f): \vec{T}(X)\to (\vec{f^2})^*\vec{T}(Y)$ (on the left hand side of (\ref{eq:coh})) is a fibre homotopy equivalence, i.e, there exists a ``homotopy inverse'' continuous map $F: (\vec{f^2})^*\vec{T}(Y)\to\vec{T}(X)$ over the identity map on $\vec{X^2}$ such that  $F\circ\vec{T}(f): \vec{T}(X)\to\vec{T}(X)$ and $\vec{T}(f)\circ F: (\vec{f^2})^*\vec{T}(Y)\to(\vec{f^2})^*\vec{T}(Y)$ are fibre homotopic to the respective identity maps.
\item A d-homotopy $H:X\times I\to Y$ is called \emph{coherently} \emph{psp} if the fibre map $\bigsqcup_t\vec{T}(H_t): \vec{T}(X)\times I\to (\vec{H^2})^*\vec{T}(Y)$ has a continuous fibre homotopy inverse
$\bar{H}: (\vec{H^2})^*\vec{T}(Y)\to\vec{T}(X)\times I$ (everything over $X\times I$).
\end{enumerate}
\end{definition}

\begin{remark}
The coherence requirement in Definition \ref{def:cpsp} is inspired by a similar requirement to what is called a directed homotopy equivalence in \cite{Goubault:17} and \cite{GFS:18}. In the context of this paper, it becomes essential in the discussion of directed topological complexity in Section \ref{s:dTC}.
\end{remark}

\begin{lemma}\label{lem:psp} (Coherently) path space preserving maps are closed under composition and partially under factorization: Let $g:X\to Y$ and $h:Y\to Z$ denote d-maps.
\begin{enumerate}
\item If $g$ and $h$ are (coherently) psp, then $h\circ g$ is so as well. 
\item If $h$ and $h\circ g$ are (coherently) psp, then $g$ is so, as well.  
\end{enumerate}
\end{lemma}

\begin{proof}
Without the coherence request, the claims follow from the 2-out-of-3 property of weak homotopy equivalences. Concerning coherence: The fibre map $\vec{T}(h):\vec{T}(Y)\to (\vec{h^2})^*\vec{T}(Z)$ and its homotopy ``inverse'' $H: (\vec{h^2})^*\vec{T}(Z) \to\vec{T}(Y)$ over $\vec{Y}^2$ pull back to homotopy inverse fibre homotopy equivalences $(\vec{g^2})^*\vec{T}(h): (\vec{g^2})^*\vec{T}(Y)\to (\vec{h^2}\circ \vec{g^2})^*\vec{T}(Z)$ and $(\vec{g^2})^*H: (\vec{h^2}\circ \vec{g^2})^*\vec{T}(Z)\to (\vec{g^2})^*\vec{T}(Y)$ over $\vec{X}^2$. 
\begin{enumerate}
\item The map $\vec{T}(h\circ g)=(\vec{g^2})^*\vec{T}(h)\circ\vec{T}(g)$ is a fibre homotopy equivalence; its homotopy inverse is the composition of  $(\vec{g^2})^*H$ with the homotopy inverse $G$ of $\vec{T}(g)$ over $\vec{X}^2$. 
\item If $HG$ denotes the fibre map homotopy inverse to $\vec{T}(h\circ g)$ over $\vec{X}^2$, then the composition of the fibre maps $(\vec{g^2})^*\vec{T}(h): (\vec{g^2})^*\vec{T}(Y)\to (\vec{h^2}\circ \vec{g^2})^*\vec{T}(Z)$ with the fibre map $HG: (\vec{h^2}\circ \vec{g^2})^*\vec{T}(Z)\to \vec{T}(X)$ -- both over $\vec{X}^2$ -- is a homotopy inverse to $\vec{T}(g): \vec{T}(X)\to (\vec{g^2})^*\vec{T}(Y)$. This can be seen as in the proof of the fact that homotopy equivalences satisfy the 2-out-of-3 property.
\end{enumerate}
\end{proof}

It is in general \emph{not} true that $g$ and $h\circ g$ psp implies $h$ psp: we obtain only information on traces in the image $g(X)$ of $g$.
 
\subsection{Inessential d-maps}\label{ss:iness} 
Now we specialize to the case of endo-d-maps $f: X\to X$ from a d-space $X$ into itself. In particular:

\begin{definition}\label{def:iness}
A d-map $f: X\to X$ is called (in the notation of Definition \ref{def:d-h})
\begin{enumerate}
\item \emph{future inessential} if there is a psp d-homotopy $id_X\nearrow f$. 
\item \emph{past inessential} if  there is a psp d-homotopy $f\nearrow id_X$.
\item \emph{neutrally inessential} if it is neutrally psp d-homotopic to $id_X$. 
\end{enumerate}
We will write $\alpha$-inessential with  $\alpha =+,-,0$ (+: future, -: past, $0$: neutral).

A d-map is called \emph{coherently} $\alpha$-inessential if, in addition, the d-homotopies to the identity map can be chosen as coherent $\alpha$-psp d-homotopies.
\end{definition}

\begin{remark}
$\alpha$-inessential d-maps correspond to a weak directed version of deformation retraction. We do not insist that the d-homotopies restrict to the identity map on a subspace. It would also be possible to relate d-spaces and d-subspaces by such a relation; compare \cite{Dubut:17}.
\end{remark}

\begin{example}
\begin{enumerate}
\item Consider the branching space $B$ from Example \ref{ex:wedge}. Every d-map $f: B\to B$ is coherently $0$-psp inessential. It is coherently $+$ (resp.\ $-$)-psp inessential if and only if $x\le f(x)$, resp.\ $f(x)<x$ for every $x\in B$.
\item A d-map $f:\partial\Box_2\to\partial\Box_2$ is coherently $0$--psp inessential if it preserves $\mb{0}, \mb{1}$ and the branches $\{ (x,y)|\; x<y\}$ and $\{ (x,y)|\; x>y\}$. It is coherently $+$ (resp.\ $-$)-psp inessential if and only if, moreover, $\mb{x}\preceq f(\mb{x})$, resp.\ $f(\mb{x})\preceq \mb{x}$ for every $\mb{x}\in \partial\Box_2$.
\item A d-map $f:\partial\Box_n\to\partial \Box_n, n>2,$ is $0$-psp inessential if it preserves all interiors of all face; cf Example \ref{ex:square}. The cases $+$, resp.\ $-$ inessential require, in addition, that $\mb{x}\preceq f(\mb{x})$, resp.\ $f(\mb{x})\preceq \mb{x}$, for every $\mb{x}\in\partial\Box_n$. 

Such a d-map $f$ is also \emph{coherently} psp inessential, as will be shown for $\alpha = +$ in the first place: Since $\mb{x}$ and $f(\mb{x})$ are contained in the same face for every $\mb{x}\in\partial\Box_n$, we may select the constant speed line d-path $\sigma_{\mb{x}}$  connecting $\mb{x}$ and $f(\mb{x})$, within $\partial\Box_n$. Together, these maps define a continuous lift $L_f: \partial\Box_n\to\vec{P}(\partial\Box_n)$ of the map $id_{\partial\Box_n}\times f: \partial\Box_n\to \vec{\partial\Box_n^2}$ into the end point map $\vec{e}_{\partial\Box_n}$. Let $s, t: \vec{P}(\partial\Box_n)\to\partial\Box_n$ denote the source and target maps (the components of the end point map $\vec{e}$). Concatenation on the right with $L_f\circ t:  \vec{P}(\partial\Box_n)\to\vec{P}(\partial\Box_n)$ maps $\vec{P}(\partial\Box_n)_{\mb{x}}^{\mb{y}}$ into $\vec{P}(\partial\Box_n)_{\mb{x}}^{f(\mb{y})}$; concatenation on the left with $L_f\circ s: \vec{P}(\partial\Box_n)\to\vec{P}(\partial\Box_n)$ maps $\vec{P}(\partial\Box_n)_{f(\mb{x})}^{f(\mb{y})}$ into $\vec{P}(\partial\Box_n)_{\mb{x}}^{f(\mb{y})}$. The diagram\\
\begin{minipage}{0.25\linewidth}
$\xymatrix{\vec{P}(\partial\Box_n)\ar[r]^{*(L_f\circ s)}\ar[d]_{f\circ} & \vec{P}(\partial\Box_n)\\
\vec{P}(\partial\Box_n)\ar[ur]_{(L_f\circ t)*} & }$
\end{minipage}\hfill
\begin{minipage}{0.7\linewidth}
commutes up to homotopy, cf \cite{Raussen:07, Raussen:18}. The map $L_f\circ t$ has a fibre homotopy inverse by mapping, for $\mb{x}\preceq \mb{y}$, a path $\sigma\in\vec{P}(\partial\Box_n)_{\mb{x}}^{f(\mb{y})}$ to the path $\sigma_f\in\vec{P}(\partial\Box_n)_{\mb{x}}^{\mb{y}}$ given by $\sigma_f(t)=\min (\mb{y},\sigma (t))$ -- well-defined since $\mb{y}$ and $f(\mb{y})$ belong to the same cube. Similarly, a fibre inverse to $L_f\circ s$ maps $\sigma\in\vec{P}(\partial\Box_n)_{\mb{x}}^{f(\mb{y})}$ to $\sigma_f\in\vec{P}(\partial\Box_n)_{f(\mb{x})}^{f(\mb{y})}$ given by $\sigma^f(t)=\max (f(\mb{x}),\sigma (t))$. Composition of $(L_f\circ t)_*$ with the fibre inverse of $*(L_f\circ s)$ yields a homotopy inverse to $\vec{P}(f)=f\circ$.
\end{minipage}

For $\alpha = -$, one may select the constant speed line path connecting $\mb{x}$ and $\mb{x}$ instead. For $\alpha =0$, connect $\mb{x}$ with $f(\mb{x})$ via $\max (\mb{x},f(\mb{x}))$ by a zig-zag of d-paths within the same cube. The diagram on the left provides a fibre homotopy inverse to the map $f\circ$.
\end{enumerate}
\end{example}

\begin{lemma}\label{lem:comp}
The family of (coherently) $\alpha$-inessential endo-d-maps is closed under composition.
\end{lemma}

\begin{proof}
By assumption, there are $\alpha$-psp-d-homotopies connecting $id_X$ with inessential d-maps $f,g: X\to X,\; \alpha = +,-,0$. Compose the homotopy connecting $id_X$ with $g$ with the map $f$ on the  right to produce an $\alpha$--d-homotopy connecting $f$ with $g\circ f$. This homotopy is psp since weak homotopy equivalences are closed under composition. Concatenate the resulting d-homotopy with the original one connecting $id_X$ with $f$ to obtain an $\alpha$-psp d-homotopy connecting $id_X$ with $f\circ g$. 

If the original d-homotopies are coherently psp, then there composition is so, as well, by Lemma \ref{lem:psp}.
\end{proof}
Inessential d-maps are, moreover, partially closed under factorization:

\begin{lemma}\label{lem:factor}
Let $g,h:X\to X$ denote endo d-maps such that $h\circ g$ and $h$ are (coherently) $0$-inessential. Then $g$ is (coherently) $0$-inessential. 
\end{lemma}

\begin{proof}
By assumption, there are dihomotopies connecting $id_X$ with $h$ resp.\ with $h\circ g$. Compose the first one with $g$ on the right resulting in a dihomotopy connecting $g$ with $h\circ g$. The dihomotopies involved induce maps $\vec{T}(H_t\circ g):\vec{T}(X)_{x_1}^{x_2}\to\vec{T}(X)_{gx_1}^{gx_2}\to\vec{T}(X)_{H_tgx_1}^{H_tgx_2}$ for  $(x_1,x_2)\in\vec{X}^2$; these maps are weak homotopy equivalences since both $H_t$ and $g$ are psp. Concatenate with the second psp dihomotopy  to connect $g$ with $id_X$. The coherent version follows from Lemma \ref{lem:psp}.
\end{proof}
Similar to the statement concluding Section \ref{ss:psp}, it is \emph{not} possible to conclude: $h\circ g$ and $g$ $0$-inessential implies $h$ is $0$-inessential; there is no information about trace spaces $\vec{T}(X)_{y_1}^{y_2}$ with (one of the) $y_i$ outside the image $g(X)$ of $g$. Inessential maps do thus \emph{not} satisfy the 2-out-of-3 property.
It is unlikely that  Lemma \ref{lem:factor} holds for $\alpha =+,-$: The proof results in a zig-zag of d-homotopies and not a simple d-homotopy starting at $id_X$, and there is no reason for a such to exist.

The following technical results will be needed in Sections \ref{ss:ri} and \ref{s:dhe}:

\begin{lemma}\label{lem:psp2}
Let $f: X\to Y$ and $g:Y\to X$ denote d-maps such that $g\circ f: X\to X$ and $f\circ g: Y\to Y$ are psp. Then the maps $f, g$ are psp each.
\end{lemma}
 
\begin{proof}
  Apply the 2-out-of-6 property for weak homotopy equivalences to the following strings of
  maps:
\[\xymatrix{\vec{T}(X)_{x_1}^{x_2}\ar[r]^{\vec{T}(f)} & \vec{T}(Y)_{f(x_1)}^{f(x_2)}\ar[r]^{\vec{T}(g)} & \vec{T}(X)_{gf(x_1)}^{gf(x_2)}\ar[r]^{\vec{T}(f)} & \vec{T}(Y)_{fgf(x_1)}^{fgf(x_2)};}\]
\[\xymatrix{\vec{T}(Y)_{y_1}^{y_2}\ar[r]^{\vec{T}(g)} & \vec{T}(X)_{g(y_1)}^{g(y_2)}\ar[r]^{\vec{T}(f)} & \vec{T}(Y)_{fg(y_1)}^{fg(y_2)}\ar[r]^{\vec{T}(g)} & \vec{T}(X)_{gfg(y_1)}^{gfg(y_2)}}.\]
\end{proof}
 
It turns out handy that inessential maps behave well under the following form of insertion:

\begin{lemma}\label{lem:insert}
Let $f: X\to Y, h: Y\to Y$ and $g: Y\to X$ denote d-maps such that $g\circ f$ is a (coherently) $\alpha$-inessential d-map on $X$ and  $h$ is a (coherently) $\alpha$-inessential d-map on $Y$. Then $g\circ h \circ f$ is a (coherently) $\alpha$-inessential d-map on $X$.
\end{lemma}
\begin{proof}
Let $H: Y\times\vec{I}\to Y$ denote a d-homotopy  connecting $id_Y$ and $h$. Then the whisker composition $gHf: X\times\vec{I}\to X,\; (x,t)\mapsto gH_t(f(x))$, is a d-homotopy between $gf$ and $ghf$. It is psp according to Lemma \ref{lem:psp2} and since weak homotopy equivalences are closed under composition. Concatenate with a psp d-homotopy connecting $id_X$ and $gf$. The coherent version of the statement is proved by using Lemma \ref{lem:psp} on top.
\end{proof}

\subsection{Rather inessential d-maps}\label{ss:ri}
The following definition takes care of the fact that inessentialness is not a 2-out-of-3 property: 
\begin{definition}\label{def:ri}
An endo d-map $h$ on $X$ is called (coherently) \emph{rather} $\alpha$-\emph{inessential} if there exists an (coherently) $\alpha$-inessential endo d-map $g$ on $X$ such that $h\circ g$ is (coherently) $\alpha$-inessential.
\end{definition}
\begin{remark}
The condition rather inessential is relaxed compared to inessential: One knows only that $\vec{T}(h)$ is a weak homotopy equivalence on trace spaces of the form $\vec{T}(X)_{gx_1}^{gx_2}$, on the image of the ``retraction'' $g$.
\end{remark}

\begin{lemma}\label{lem:comp2}
The family of (coherently) rather $\alpha$-inessential d-maps is closed under composition.
\end{lemma}
\begin{proof}
For rather $\alpha$-inessential d-maps $g_i$ there exist $\alpha$-inessential d-maps $h_i$ such that $g_i\circ h_i$ is $\alpha$-inessential. The endo d-map $g_1\circ g_2\circ h_2\circ h_1$ is $\alpha$-inessential by Lemma \ref{lem:insert}, and so is $h_2\circ h_1$ by Lemma \ref{lem:comp}.
\end{proof}

The following property shows that the family of rather inessential maps has better properties than the family of inessential maps. 
\begin{lemma}\label{lem:2-3}
Let $g,h:X\to X$ denote d-maps such that $g\circ h$ is (coherently) rather $\alpha$-inessential.
\begin{enumerate}
\item If $g$ is (coherently) $\alpha$-rather inessential, then $h$ is (coherently) rather $0$-inessential.
\item If $h$ is (coherently) $\alpha$ rather inessential, then so is $g$.
\end{enumerate}
\end{lemma}

\begin{proof}
\begin{enumerate}
\item By assumption, there exist inessential d-maps $k,l: X\to X$ such that $g\circ h\circ k$ and $g\circ l$ is inessential. By Lemma \ref{lem:insert}, the map $(gl)\circ (hk)=g\circ l\circ (hk)$ is inessential. Since $g\circ l$ is inessential, by Lemma \ref{lem:factor}, $h\circ k$ is $0$-inessential, and hence $h$ is rather $0$-inessential.
\item By assumption, there exist inessential d-maps $k,l: X\to X$ such that $g\circ h\circ k$ and $h\circ l$ are inessential. By Lemma \ref{lem:comp}, resp.\ Lemma \ref{lem:insert}, also the maps $h\circ l\circ k$ and $g\circ (hlk)=(gh)\circ l\circ k$ are inessential. Hence $g$ is rather inessential.
\end{enumerate}
\end{proof}

\begin{remark}
\begin{enumerate}
\item Lemma \ref{lem:comp2} and \ref{lem:2-3} together yield: The  rather $0$-inessential endo-d-maps on a d-space $X$ enjoy the 2-out-of-3 property.
\item In Lemma \ref{lem:2-3}(1), we can in general not conclude that $h$ is rather  $\alpha$-inessential. It is not certain that there exists a future (or past) d-homotopy relating $id_X$ and $h$.
\item It is possible to refine (1): If $g$ and $g\circ h$ are rather  $0$-inessential through a d-homotopy (instead of a dihomotopy, cf Definition \ref{def:d-h}(2)) from $id_X$, then $h$ is so, as well.
\end{enumerate}
\end{remark}
\subsection{Generalization to monoids}\label{ss:mono}
Considering the monoid $\vec{C}(X,X)$ of d-maps on $X$, we may localize the submonoid consisting of $\alpha$-inessential d-maps (and coherent versions), and declare, for inessential d-maps $g,h$ on $X$,  the d-maps $f$ and $h\circ f\circ g$ \emph{equivalent mod inessentials}. In particular, inessential maps are equivalent to the identity $id_X$, and hence so are rather inessential maps. Furthermore, we will in Section \ref{s:pcc} consider the quotient monoid with respect to the symmetric and transitive closure (an equivalence relation) of this relation.

The notions ``inessential'' and ``rather inessential'', for $\alpha = 0$, in the preceeding sections are special cases of the following picture comparing a monoid $(\mb{M},\circ )$ to a submonoid $\mb{S}\subset\mb{M}$ enjoying the ``inessentiality'' property (compare with Lemma \ref{lem:factor}): 
\begin{equation}\label{eq:iness}
g\in\mb{M}, h, h\circ g\in\mb{S}\Rightarrow g\in\mb{S}.\
\end{equation}
In such a situation, one may define a ``closure'' $\bar{\mb{S}}$ of $\mb{S}$ by $\bar{\mb{S}}:=\{ h\in\mb{M}|\;\exists g\in\mb{S}: h\circ g\in\mb{S}\}$. Using the same formal arguments as in Lemma \ref{lem:comp2} and Lemma \ref{lem:2-3}, one proves:

\begin{proposition}
Let  $\mb{S}\subset\mb{M}$ denote a pair of monoids enjoying the inessentiality property (\ref{eq:iness}). Then 
\begin{enumerate}
\item $\bar{\mb{S}}$ is a submonoid of $\mb{M}$ containing $\mb{S}$.
\item $\bar{\mb{S}}$ enjoys the 2-out-of-3-property, i.e., if two of the three elements $g, h, g\circ h$ are contained in $\bar{\mb{S}}$, then so is the third.
\end{enumerate}
\end{proposition}

\section{Directed homotopy equivalences}\label{s:dhe}
\subsection{Definitions and Examples}\label{ss:defex}
\begin{definition}\label{def:dhe}
A d-map $f:X\to Y$ is called a (coherent) $\alpha$ directed homotopy equivalence if there exists a d-map $g:Y\to X$ such that $g\circ f: X\to X$ and $f\circ g: Y\to Y$ are (coherently) $\alpha$-rather inessential.
\end{definition}

In particluar, a directed homotopy equivalence is an ordinary homotopy equivalence. Moreover, an $\alpha$-directed homotopy equivalence preserves trace spaces ``up to inessentials'':
\begin{lemma}\label{lem:F-eq}
Let $f: X\to Y$ denote an $\alpha$-directed homotopy equivalence. Then there exists an inessential d-map $h: X\to X$ such that $\vec{T}(f): \vec{T}(X)_{hx_1}^{hx_2}\to\vec{T}(Y)_{fhx_1}^{fhx_2}$ is a weak homotopy equivalence for all $(x_1,x_2)\in\vec{X}^2$ (``$f$ is psp on the image of $h$'') 
\end{lemma}
\begin{proof}
By definition, there exist inessential maps $h_X: X\to X$ and $h_Y: Y\to Y$ such that $g\circ f\circ h_X$ and $f\circ g\circ h_Y$ are inessential. By Lemma \ref{lem:insert}, the maps $
g\circ h_Y\circ (f\circ h_X)$ and $f\circ h_X\circ (g\circ h_Y)$ are inessential, as well. Apply Lemma \ref{lem:psp2} to the maps $f\circ h_X$ and $g\circ h_Y$ to conclude that the map $f\circ h_X$ is psp. Since $h_X$ is psp, the maps induced by $f$ on trace spaces in the image of $h_X$ are weak homotopy equivalences.
\end{proof}

\begin{example}\label{ex:wedge2}
\begin{enumerate}
\item Consider the branching space $B$ from Example \ref{ex:wedge}, and let us see that the inclusion $i: O\to B$ of the one-point d-space $O$ consisting of the origin is \emph{not} a $+$-equivalence: Let $c: B\to O$ denote the unique map. For any ($+$-inessential) map $h: B\to B$, the composition $i\circ c\circ h=i\circ c$, and $i\circ c=c_O: B\to B$ maps every element to $O$. We have seen in Example \ref{ex:wedge} that this map is \emph{not} future d-homotopic to the identity map on $B$.\\
Remark that the map $i$ is a past (and thus a neutral) directed homotopy equivalence, even coherently so: The map $c\circ i: O\to O$ is the identity map. Let $\sigma_O$ denote the constant (and unique) d-path at $O$. Then $\vec{T}(c_O): \vec{T}(B)\to\vec{c_O^2}\vec{T}(B)=\vec{B^2}\times \{\sigma_O\}$ is a fibre  homotopy equivalence over $\vec{B}^2$. The fibre inverse maps $((x,y);\sigma_O)\in \vec{B^2}\times \{ \sigma_O\}$ to the unique trace from $x$ to $y$ in $B$.  
\item Let $W$ denote the geometric realization of the directed graph (``the letter W'')
\[\xymatrix{\cdot_A & \cdot_B\ar[l]\ar[r] & \cdot_C & \cdot_D\ar[l]\ar[r] & \cdot_E}.\] 
All non-empty path and trace spaces are contractible, and the unique d-map $c:W\to O$ to the one point/one path d-space $O$ is coherently psp (cf Definition \ref{def:cpsp}).
To check, whether $W$ is dicontractible (ie $0$-directed homotopy equivalent to $O$), let $i: O\to W$ and $h: W\to W$ denote arbitrary d-maps. The composite map $i\circ c\circ h = i\circ c: W\to W$ is constant and thus not (neutrally) d-homotopic to $id_W$ since any dihomotopy $H$ with $H_0=id_W$ will fix the branch points $B, C$ and $D$.
\item \begin{minipage}{0.75\linewidth}Let us compare the spaces $X=\partial\Box_2$ and $Y=\Box_2\setminus ]\mb{j_1},\mb{j_2}[$ with $\mb{j_1}=(\frac{1}{3},\frac{1}{3})$ and $\mb{j_2}=(\frac{2}{3},\frac{2}{3})$; the latter space occurs as a simple model for mutual exclusion between two processes  in concurrency theory. 
\end{minipage}\hfill
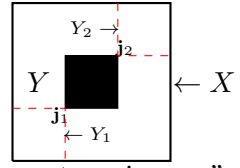
\begin{minipage}{0.2\linewidth}
\begin{tikzpicture}
\draw[thick] (0,0) -- (2.1,0) -- (2.1,2.1) -- (0,2.1) -- (0,0);
\draw[fill] (0.7,0.7) -- (1.4,0.7) -- (1.4,1.4) -- (0.7,1.4) -- (0.7,0.7);
\node at (2.55,1.05) {$\leftarrow X$};
\node at (0.35,1.05) {$Y$};
\node at (0.65,0.6) {\tiny $\mb{j}_1$};
\node at (1.5,1.5) {\tiny $\mb{j}_2$};
\draw[red,dashed] (0,0.7) -- (0.7,0.7) -- (0.7,0);
\draw[red,dashed] (1.4,2.1) -- (1.4,1.4) -- (2.1,1.4);
\node at (1,0.35){\tiny $\leftarrow Y_1$};
\node at (1.1,1.75){\tiny $Y_2\to$};
\end{tikzpicture}
\end{minipage}

\noindent Assume $f: X\to Y$ is an $\alpha$ directed homotopy equivalence with ``homotopy inverse'' $g: Y\to X$. By Lemma \ref{lem:F-eq}, 
$f(\mb{0})\in [\mb{0},\mb{j_1}]$ and $f(\mb{1})\in [\mb{j_2},\mb{1}]$, whereas $f(\partial\Box_2\setminus\{\mb{0},\mb{1}\})\subseteq\Box_2\setminus ([\mb{0},\mb{j_1}]\cup [\mb{j_2},\mb{1}])$. Since $f$ is continuous, we conclude that $f(\mb{0})\in Y_1:=\{(x,y)\in\Box_2|\; x=\frac{1}{3}\mbox{ or }y=\frac{1}{3}\}$ and $f(\mb{1})\in Y_2:=\{(x,y)\in\Box_2|\;x=\frac{2}{3}\mbox{ or }y=\frac{2}{3}\}$. Similarly, $g([\mb{0},\mb{j_1}])=\{ \mb{0}\}$ and $g([\mb{j_2},\mb{1}])=\{ \mb{1}\}$. In particular, $f\circ g$ maps $[\mb{0},\mb{j_1}]$ into $Y_1$ and $[\mb{j_2},\mb{1}]$ into $Y_2$. This map can therefore neither be future- nor past-d-homotopic to the identity map.

Otherwise for $\alpha=0$: The map $f=f_1\times f_1: X\to Y, f_1(x)=\frac{x+1}{3}$, is a \emph{neutral} directed homotopy equivalence with homotopy inverse $g=g_1\times g_1: Y\to X,\; g_1(y)=\begin{cases} 0 & y\le \frac{1}{3}\\ 3(y-\frac{1}{3}) & \frac{1}{3}\le y\le\frac{2}{3}\\ 1 & y\ge\frac{2}{3}\end{cases}$.
To check coherence, observe that $g\circ f=id_X$ whereas\\ $f_1\circ g_1(y)=\begin{cases} \frac{1}{3}& y\le \frac{1}{3}\\ y & \frac{1}{3}\le y\le\frac{2}{3}\\ \frac{2}{3} & y\ge\frac{2}{3}\end{cases}$,
and $f\circ g$ contracts $Y$ to the inner hollow square. 

The trace space $\vec{T}(Y)_{fg\mb{x}}^{fg\mb{y}}$ connecting two points in that hollow square has two elements if and only if $\mb{x}\in [\mb{0},\mb{j_0}]$ and $\mb{y}\in [\mb{j_1},\mb{1}]$ and one element else. A fibre homotopy inverse to $\vec{T}(Y)\to (\vec{f^2}\circ \vec{g^2})^*\vec{TY}$ associates to such a trace the trace connecting $\mb{x}$ with $\mb{y}$ consisting of a horizontal and a vertical path; the order (first horizontal, then vertical or the reverse) follows the order of the trace connecting $fg\mb{x}$ and $fg\mb{y}$.

\item The directed circle $\vec{S}^1$ is the pre-cubical set with one $0$-cell and one $1$-cell. Its geometric realization $|\vec{S}^1|$ is a circle on which directed paths proceed counter-clockwise, ie., they are images of non-decreasing paths under the universal covering $\exp : \vec{\mb{R}}\to |\vec{S}^1|$. The directed $n$-torus $\vec{T}^n=(\vec{S}^1)^n,\; n\in\mb{N},$ arises as product of $n$ directed circles.

It is shown in \cite[Section 4.4]{Raussen:18} that only directed \emph{homeomorphisms} on $\vec{S}^1$ are inessential; furthermore, the inessential maps on a directed torus are necessarily products of directed homeomorphisms on each factor. In particular, these inessential maps are bijections; hence, rather inessential maps are necessarily inessential.  As a result, only very few maps are directed self-homotopy equivalences: On a directed circle, these are only the directed homeomorphisms, and on a torus, these are only products of such.
\end{enumerate}
\end{example}

\begin{definition}\label{def:dhet}
Two d-spaces $X$ and $Y$ are called (coherently) $\alpha$-equivalent if there is a (coherent) $\alpha$ directed homotopy equivalence $f:X\to Y$.
\end{definition}

\begin{remark}\label{rem:GFS}
\begin{enumerate}
\item It follows from Definition \ref{def:dhe} and from Lemma \ref{lem:comp2} that (coherent) $\alpha$-equivalence is an equivalence relation between d-spaces.
\item For two $\alpha$-equivalent d-spaces $X$ and $Y$, there exist in fact d-maps $\bar{f}: X\to Y$ and $\bar{g}: Y\to X$ such that $\bar{g}\circ\bar{f}: X\to X$ and $\bar{f}\circ\bar{g}: Y\to Y$ are (coherently) $\alpha$-inessential (not just rather inessential). Here is why: If there are d-maps $f: X\to Y, g: Y\to X,$ and inessential d-maps $h: X\to X$ and $k: Y\to Y$ such that $g\circ f\circ h: X\to X$ and $f\circ g\circ k: Y\to Y$ are $\alpha$-inessential, then replace $f$ by $\bar{f}=f\circ h$ and $g$ by $\bar{g}=g\circ k$. The map $\bar{g}\circ \bar{f} = g\circ k\circ (f\circ h): X\to X$ is $\alpha$-inessential according to Lemma \ref{lem:insert}. Likewise $\bar{f}\circ\bar{g} : Y\to Y$.

We need to come back to the weaker notion from Definition \ref{def:dhe} in Section \ref{ss:2-3} in order to establish a 2-out-of-3 property for \emph{individual} $\alpha$-directed homotopy equivalences, cf Proposition \ref{prop:2-3}.
\item Goubault \cite{Goubault:17} and Goubault, Farber and Sagnier \cite{GFS:18} propose the following requirements to a d-map $f:X\to Y$ to qualify as a directed homotopy equivalence:
\begin{enumerate}
\item $f$ is an ordinary homotopy equivalence with a homotopy inverse d-map $g: Y\to X$.
\item Both $\vec{T}(f):\vec{T}(X)\to (\vec{f^2})^*\vec{T}(Y)$ and $\vec{T}(g):\vec{T}(Y)\to (\vec{g^2})^*\vec{T}(X)$ are fibre homotopy equivalences.
\end{enumerate}
(b) is a reformulation of the original wording in \cite{Goubault:17} and \cite{GFS:18}. Remark that this definition has both weaker and stronger assumptions compared to our definition. Weaker since it is not assumed that the compositions $g\circ f$, resp.\ $f\circ g$ are \emph{d}-homotopic to the identity maps (in any flavour). Stronger since no compositions with inessential maps are used.  

With their definition, $i: O\to B$ from Example \ref{ex:wedge2}(1) is always a directed homotopy equivalence with homotopy inverse $c$. Even $i: O\to W$ from Example \ref{ex:wedge2}(2) does then qualify as a directed homotopy equivalence; coherence (ie (b) above) is established as in Example \ref{ex:wedge2}(2). 
\end{enumerate}
\end{remark}

\subsection{A 2-out-of-3 property}\label{ss:2-3}
Previous attempts to define directed homotopy equivalences failed to establish the 2-out-of-3 property among morphisms (d-maps) in the category $\mathbf{dTop}$; cf Definition \ref{def:d}(4). Proposition \ref{prop:2-3} below shows that Definition \ref{def:dhe} is suitably weak to establish such a property.

We will need the following generalization of Lemma \ref{lem:insert} (for all flavours $\alpha$):
\begin{lemma}\label{lem:insert2}
Let $f: X\to Y$ and $g: Y\to X$ denote d-maps such that $g\circ f: X\to X$ and $f\circ g: Y\to Y$ are rather inessential. Then, for every rather inessential map $l: Y\to Y$, the composition $g\circ l\circ f$ is rather inessential.
\end{lemma}

\begin{proof}
By assumption, there exist inessential d-maps $h: X\to X, k,m: Y\to Y$ such that $g\circ f\circ h: X\to X$, $f\circ g\circ k: Y\to Y$ and $l\circ m$ are inessential. By Lemma \ref{lem:comp}, $k\circ m: Y\to Y$ is inessential. By Lemma \ref{lem:insert}, $g\circ (km)\circ (fh): X\to X$ is inessential, and so is $l\circ (fgk)\circ m: Y\to Y$. Again by Lemma \ref{lem:insert}, the map $(glf)\circ (gkmfh)=g\circ (lfgkm)\circ (fh)$ is inessential. Hence $g\circ l\circ f$ is rather inessential.
\end{proof}
   
\begin{proposition}\label{prop:2-3}
The $0$-directed homotopy equivalences satisfy the 2-out-of-3 property among morphisms (d-maps) in the category $\mathbf{dTop}$.
\end{proposition}
\begin{proof} We try the 2-out-of 3 property for all flavours $\alpha$. Only for one of the factorization cases it is necessary to assume that $\alpha =0$:
\begin{description}
\item[Composition] Let $f_1: X\to Y, f_2:Y\to Z$ denote $\alpha$-directed homotopy equivalences with ``inverses'' $g_1:Y\to X, g_2:Z\to Y$: Both $g_1\circ f_1$ and $g_2\circ f_2$ are then rather inessential. By Lemma \ref{lem:insert2}, so is their composition
$(g_1\circ g_2)\circ (f_2\circ f_1)=g_1\circ (g_2\circ f_2)\circ f_1$. Similarly for composition in the reverse order. 
 \item[Factorization 1] Let $f_1: X\to Y$ and $f_2: Y\to Z$ be d-maps such that $f_1$ and $f_2\circ f_1$ are $\alpha$-directed homotopy equivalences.
By assumption, there exist reverse maps $g_1: Y\to X$ and $g_{12}: Z\to X$ such that $g_1\circ f_1: Y\to Y$ and $f_1\circ g_1, g_{12}\circ f_2\circ f_1 : X\to X$ and $f_2\circ f_1\circ g_{12}: Z\to Z$ are rather inessential. To check $f_1\circ g_{12}$ as a ``left inverse'' to $f_2$, we consider:
$f_1\circ g_{12}\circ f_2\circ (f_1\circ g_1)=f_1\circ (g_{12}\circ f_2\circ f_1)\circ g_1$ that is rather inessential by Lemma \ref{lem:insert2}. By Lemma \ref{lem:2-3}, the composition $(f_1\circ g_{12})\circ f_2$ is so, as well. 
\item[Factorization 2] $f_2$ and $f_2\circ f_1$ are $\alpha$-directed homotopy equivalences. Hence there exist reverse maps $g_2: Z\to Y$ and $g_{12}: Z\to X$ such that $g_2\circ f_2: Y\to Y$ and $f_2\circ g_2, f_2\circ f_1\circ g_{12}: Z\to Z$ and $g_{12}\circ f_2\circ f_1: X\to X$ are rather inessential. Hence
$(g_2\circ f_2)\circ f_1\circ g_{12}\circ f_2=g_2\circ (f_2\circ f_1\circ g_{12})\circ f_2: Y\to Y$ is rather inessential by Lemma \ref{lem:insert2}. By Lemma \ref{lem:2-3}, $f_1\circ g_{12}\circ f_2$ is rather $0$-inessential. 
\end{description} 
\end{proof}

\begin{remark}
Proposition \ref{prop:2-3} is also valid for \emph{coherent} $0$-directed homotopy equivalences. Every lemma used in the proof has also a coherent version.
\end{remark}

\subsection{Generalizations}
\subsubsection{$\mc{F}$-equivalences}\label{sss:F-eq}
The framework presented in the previous sections can be generalized to compare d-spaces that are equivalent in a weaker sense. The notion of (weak) homotopy equivalence can be replaced by that of an $\mc{F}$-equivalence:

\begin{definition}\cite[Definition 2.5]{Ziemianski:18}\label{def:F}
A familiy $\mc{F}$ of $\mathbf{Top}$-morphisms is called an \emph{equivalence system} if 
\begin{enumerate}
\item $\mc{F}$ is closed under homotopy;
\item $\mc{F}$ satisfies the 2-out-of-3-property;
\item $\mc{F}$ contains all weak homotopy equivalences:
\item An $\mc{F}$-morphism $f$ induces a bijection $\pi_0(f)$ on sets of path components.
\item $\mc{F}$ is closed under finite products and finite sums. A finite sum is contained in $\mc{F}$ if and only if every summand is so.
\end{enumerate}
\end{definition}

Many equivalence systems arise from a functor $\mc{G}: \mathbf{Top}\to \mc{C}_{\mc{F}}$ into a category $\mc{C}_{\mc{F}}$: The family $\mc{F}$ consists then of those morphisms $f$ for which $\mc{G}(f)$ is an isomorphism. This is the case for 
the most important examples of equivalence systems for our purposes that are collected in the list below: 
\begin{description}
\item[$\mc{F}_{\infty}$] the weak homotopy equivalences
\item[$\mc{F}_0$] the family of maps inducing bijections on sets of path components
\item[$\mc{F}_k$] the family of maps inducing bijections on $\pi_0$ and isomorphisms on all homotopy groups $\pi_n$ for every $n\le k$ and every choice of basepoint
\item[$\mc{H}_{A,k}$] ($A$ an abelian group): the family of maps inducing isomorphisms on all homology groups $H_n(-;A)$ for every  $n\le k$.
\end{description}

The whole program in the previous sections can be adapted: We will call a d-map $\mc{F}$-path space preserving 
if $\vec{T}(f)_{x_1}^{x_2}: \vec{T}(X)_{x_1}^{x_2}\to\vec{T}(Y)_{fx_1}^{fx_2}$ is an $\mc{F}$-equivalence for all $(x_1,x_2)\in\vec{X}^2$. This leads to the concept of (coherently) $\alpha\mc{F}$-inessential endo-d-maps (generalising Definition \ref{def:iness}) and (coherent)  $\alpha\mc{F}$-equivalence (generalising Definition \ref{def:dhe}). If the system $\mc{F}$ satisfies the 2-out-of-6 property (as all those in the list above), then all properties of inessential maps and directed homotopy equivalences have counterparts ($\mc{F}$-(rather) inessential maps, resp.\ $\mc{F}$-equivalences) in this framework.

Directed maps induce functors between the ``natural homology'' systems \cite{DGG:15} associated to d-spaces. If  a d-map $f:X\to Y$ is an $\alpha\mc{H}_{A,k}$-equivalence, then these natural homology systems become isomorphisms on the level of homology ``up to $\alpha\mc{H}_{A,k}$-inessential  maps'' on both sides. Details will be developed in Section \ref{s:pcc}.

\subsubsection{General endomorphisms}
The pattern of proof in Lemma \ref{lem:insert2} can be generalized to the following setting: Let $F: A\to B$ and $G:B\to A$ denote morphisms in some given category $\mc{C}$. Let $Endo(A)=Mor (A,A)$ and $Endo(B)=Mor(B,B)$ denote the monoids of endomorphisms. Let $S_A\subseteq Endo (A), S_B\subseteq Endo (B)$ denote submonoids enjoying the inessentiality property (\ref{eq:iness}) and giving rise to closures $S_A\subseteq\bar{S}_A\subseteq Endo(A)$, resp.\ $S_B\subseteq\bar{S}_B\subseteq Endo(B)$, cf Section \ref{ss:mono}. 
Together, the mophisms $F,G$ define maps (not monoid morphisms, in general!) $(F,G)_{\#}: Endo(A)\to Endo (B),\; h\mapsto F\circ h\circ G$ and $(G,F)_{\#}: Endo(B)\to Endo (A),\; k\mapsto G\circ k\circ F$. 
\begin{lemma}\label{lem:insertm}
Suppose $F\circ G=(F,G)_{\#}(id_A)\in\bar{S}_B$ and $G\circ F=(G,F)_{\#}(id_B)\in\bar{S}_A$. Then $(F,G)_{\#}(\bar{S_A})\subseteq\bar{S}_B$.
\end{lemma}
Let us call a morphism $F:A\to B$ an equivalence up to inessentials if there exists a morphism $G:B\to A$ satisfying the assumptions of Lemma \ref{lem:insertm}. Then we can prove, analogously to Proposition \ref{prop:2-3}:
\begin{proposition}\label{prop:2-3m}
The property ``equivalence up to inessentials'' among $\mc{C}$-morphisms is a 2-out-of-3-property.
\end{proposition}

\section{Directed homotopy equivalences and directed topological complexity}\label{s:dTC}
\subsection{(Directed) topological complexity}
Michael Farber \cite{Farber:03} introduced the notion ``topological complexity'' $T(X)$ of a topological space $X$ to address the motion planning problem: Given a topological space $X$, find the minimal number of subspaces $Y_i$ covering $X\times X$ such that there exists a continuous section of the path space fibration $e_X: X^I\to X\times X$ over each $Y_i$; technically speaking, this is the Schwartz genus of the path space fibration. This notion has been extensively studied, and the topological complexity has been determined or at least estimated for many spaces of interest in the literature.

Recently, Goubault, Farber, and Sagnier \cite{GFS:18} have proposed the following definition for directed topological complexity (adapted from one of the customary definitions of $TC(X)$ to the directed framework, cited here using the notations from the present paper) for a d-space $(X, \vec{P}(X))$ such that $X$ is a Euclidean Neighbourhood Retract (ENR):

\begin{definition}\cite[Definition 3]{GFS:18}
The \emph{directed topological complexity} $\overrightarrow{TC}(X,\vec{P}(X))$ of a d-space $(X,\vec{P}(X))$ is the minimal number $n\in\mb{N}$ (if existing) such that there is a (not-necessarily continuous) section $s$ of the fibre map $\vec{e}_X: \vec{T}(X)\to \vec{X}^2$ and $\vec{X}^2=\bigsqcup_1^nX_i$ over mutually disjoint subspaces $X_i\subseteq \vec{X}^2$ and such that $s|_{X_i}$ is continuous.
\end{definition}

The authors then calculate the directed topological complexity of some important d-spaces, but it seems fair to say that the theory is not well-developed yet in the directed case. The crucial complication compared with the non-directed case stems from the fact that the end-point map $\vec{e}_X: \vec{T}(X)\to \vec{X}^2$ is, in most cases, not a fibration: In general, the homotopy types of fibres varies with the end points; d-spaces are not homogeneous! 

The following recent result by Borat and Grant \cite{BG:18} concerning the ``directed spheres'' $\partial\Box_n$, the boundaries of unit cubes, came as a surprise: Whereas the (ordinary) topological complexity of spheres $S^n$ is $2$ for $n$ odd and $3$ for $n$ even, they show:

\begin{proposition}\cite{BG:18}
Directed topological complexity satisfies: $\overrightarrow{TC}(\partial\Box_n)=2$ for all $n\ge 2$.
\end{proposition}

\subsection{Invariance under directed homotopy equivalence}
Goubault, Farber and Sagnier investigate also the invariance of topological complexity under their notion of directed homotopy equivalence, cf the final remark in Section \ref{ss:defex}. It is quite easy to check that the following result \cite[Proposition 7]{GFS:18} also holds when using our notion of coherent $\alpha$ directed homotopy equivalence,  regardless of the flavour $\alpha$:

\begin{proposition}
Directed topological complexity is invariant under directed homotopy equivalence. 
\end{proposition}

\begin{proof}
The basic argument in the proof is essentially the one given in \cite{GFS:18}. We adapt it to our setting as follows: Assume $f: X\to Y$ and $g: Y \to X$ are d-maps such that $g\circ f$ and $f\circ g$ are coherently inessential, cf Definition \ref{def:dhet}. Then $\vec{T}(g\circ f): \vec{T}(X)\to (\vec{g^2}\circ\vec{f^2})^*\vec{T}(X)$ is a fibre homotopy equivalence over $\vec{X}^2$. Let $s: \vec{Y}^2\to\vec{T}(Y)$ denote a section of $\vec{e}_Y$, continuous on $n$ disjoint subspaces $Y_i\subset\vec{Y}^2$ covering $\vec{Y}^2$. Then the composite $\vec{T}(g)\circ s\circ \vec{f^2}: \vec{X}^2\to\vec{TX}$ is a map over $(\vec{g^2}\circ\vec{f^2}): \vec{X}^2\to\vec{X}^2$, continuous on the $n$ disjoint subspaces $X_i:= f^{-1}(Y_i)\subset\vec{X}^2$ covering $\vec{X}^2$. It pulls back to a section of $(\vec{g^2}\circ\vec{f^2})^*\vec{T}(X)$ over $\vec{X}^2$. Composition with a fibre inverse to $\vec{T}(g\circ f)$ yields a section of $\vec{e}_X: \vec{T}(X)\to\vec{X}^2$ the restrictions of which to the subsets
$X_i$ are continuous.

Hence $\overrightarrow{TC}(X)\le\overrightarrow{TC}(Y)$. A symmetric argument proves the reverse inequality. 
\end{proof}

\section{Directed homotopy equivalences and pair component categories.}\label{s:pcc}

\subsection{Extension categories and relatives}
It was the main aim of \cite{Raussen:18} (consult that paper for motivations and details) to organize essential information about the relations between spaces of directed paths (with varying end points) \emph{within a given d-space} in appropriate categories; to investigate, and if possible, to simplify and/or to compress these categories. Point of departure is the extension category $E\vec{\pi}_1(X)$ of the fundamental category \cite{Grandis:01,FGHMR:16} $\vec{\pi}_1(X)$ (objects: points in $X$; morphisms: homotopy classes of d-paths between source and target). $E\vec{\pi}_1(X)$ contains as objects the reachable pairs in $\vec{X}^2$. (Extension) morphisms have the form $(\sigma * ,*\tau )\in E\vec{\pi}_1(X)((x,y),(x',y'))=\vec{\pi}_1(X)_{x'}^x\times\vec{\pi}_1(X)_y^{y'}$ -- contravariant in the first coordinate and covariant in the second. To this category, further morphisms arising from endo-d-maps $f:X\to X$ may be added ``acting'' on those already present: 

An endo-d-map $f$ gives rise to a morphism from $(x,y)$ to $(fx,fy)$. All endo-d-maps give rise to the morphisms of a category $d(X)$ (under composition) on the object set $\vec{X^2}$. The ``mixed'' category $dE\vec{\pi}_1(X)$, cf \cite{Raussen:18}, has again $\vec{X^2}$ as set of objects; its morphisms are freely generated by those of  $E\vec{\pi}_1(X)$ and of $d(X)$ modulo the relations generated by the following two commutative diagrams 

\begin{equation}\label{eq:dE}
\xymatrixcolsep{5pc}\xymatrix{(fx,fy)\ar[r]^{((f\circ\sigma)*, *(f\circ\tau ))} & (fx',fy')\\
(x,y)\ar[u]_f\ar[r]^{(\sigma *, * \tau )} & (x',y')\ar[u]_f},
\end{equation}
for any endo d-map $f:X\to X, (x,y)\in\vec{X^2}, \sigma\in\vec{\pi}_1(X)_{x'}^x, \tau\in\vec{\pi}_1(X)_y^{y'}$; and

\begin{equation}\label{eq:flow}
\xymatrix{(fx, fy)\ar[r]^{*H(y)} & (fx, gy)\\
(x,y)\ar[u]^f\ar[r]^g & (gx,gy)\ar[u]_{H(x)*}}
\end{equation}
for any future d-homotopy (cf Definition \ref{def:d-h}) $H: X\times\vec{I}\to X$ from  $H_0=f$ to $H_1=g$, and for all $(x,y)\in \vec{X^2}$. Here $H(x)$ is the d-path arising by restricting $H$ to $x\in X$.

From now on, we will make use of general $\mc{F}$-equivalences, cf Section \ref{sss:F-eq}; at a first read just think of $\mc{F}=\mc{F}_{\infty}$ of consisting of the weak homotopy equivalences. The category $dE\vec{\pi}_1(X)$ comes equipped with a functor into any of the categories $\mc{C}_{\mc{F}}$ corresponding to the equivalence system $\mc{F}$: Pairs of points $(x,y)$ go to trace space $\vec{T}(X)_x^y$ resp.\ its $\mc{C}_{\mc{F}}$-invariant (eg homology); extensions and d-maps to the morphism induced by those on (the invariants of) these trace spaces. The diagrams associated to (\ref{eq:dE}) and to (\ref{eq:flow}) in $\mc{C}_{\mc{F}}$ do commute!

Instead of letting \emph{all} endo-d-maps act, we may restrict to adding \emph{only} the $\alpha\mc{F}$-\emph{inessential} ones, corresponding to the symmetries of the d-space $X$: First of all, they give rise to ``inessential'' subcategories $\Sigma_{\mc{F}}^{\alpha}(X)\subset d(X)$ restricting the morphisms in $d(X)$ to those arising from  $\alpha\mc{F}$ inessential d-maps. Combining with the extension category $E\vec{\pi}_1(X)$ -- as above -- yields a subcategory $\Sigma_{\mc{F}}^{\alpha}E\vec{\pi}_1(X)\subset dE\vec{\pi}_1(X),\; \alpha = +, -, 0$. Remark that inessential d-maps are mapped into \emph{iso}morphisms in $\mc{C}_{\mc{F}}$ under the functor mentioned above.

Inverting all symmetries given by inessential morphisms $f\in \Sigma_{\mc{F}}^{\alpha}(X)$ gives rise to localized categories 
$\Sigma_{\mc{F}}^{\alpha}E\vec{\pi}_1(X)[\Sigma_{\mc{F}}^{\alpha}(X)^{-1}]$ -- relating objects $(x,y)$ and $(fx,fy)$ resp.\ extensions $(\sigma *,*\tau): (x,y)\to (x',y')$ and $(f\sigma *,*f\tau): (fx,fy)\to (fx',fy)$ by an \emph{iso}morphism. Remark that also a \emph{rather} inessential morphism $h$ gives rise to isomorphisms in the localized category: If $g$ and $h\circ g$ are inessential morphisms, then $h=(h\circ g)\circ g^{-1}$ is a product of isomorphisms (with inverse $g\circ (h\circ g)^{-1}$). Note also that the functors from above into $\mc{C}_{\mc{F}}$ extend to the localized categories (eg to $\mathbf{Ho Top}$) since invertible morphisms go to isomorphisms. 

The path components of $\vec{X}^2$ with respect to isomorphisms in such a localized category (with respect to zig-zag paths induced by inessential morphisms) form the objects of the quotient \emph{pair component  categories} $\vec{\pi}_0(X; \alpha , \mc{F})_*^*$; cf.\ \cite{Raussen:18}. Remark that rather inessential morphisms in $\Sigma^{\alpha}_{\mc{F}}(X)$ become identities in this quotient category.

\subsection{Pair component categories as obstructions to directed homotopy equivalences}
A d-map $F:X\to Y$ induces a functor $E\vec{\pi}_1(F): E\vec{\pi}_1(X)\to E\vec{\pi}_1(Y)$: To $(x_1,x_2)\in\vec{X}^2$ we associate the pair $(Fx_1,Fx_2)$ and to an extension by $\sigma\in\vec{\pi}_1(X)((x_1,x_2),(x_1',x_2'))$ we associate the extension by $F\sigma\in\vec{\pi}_1(Y)((Fx_1,Fx_2),(Fx_1',Fx_2'))$. There is no obvious way to construct a functor relating endo-d-maps on $X$ to endo-d-maps on $Y$ on the basis of a \emph{single} d-map $F: X\to Y$, such. That option seems to arise if one considers a \emph{pair} of d-maps $F:X\to Y$ and $G:Y\to X$. We may then associate to an endo-d-map $f:X\to X$ the endo-d-map $F\circ f\circ G:Y\to Y$. But this construction is not functorial: To a composition $f_2\circ f_1$ of d-maps $f_1,f_2: X\to X$, we associate $F\circ f_2\circ f_1\circ G: Y\to Y$ and not $F\circ f_2\circ G\circ F\circ f_1\circ G$.

Now assume that d-spaces $X$ and $Y$ are $\alpha\mc{F}$-equivalent, ie there exists a d-map $F: X\to Y$ with ``$\alpha\mc{F}$ inverse'' $G: Y\to X$. By Lemma \ref{lem:insert} and Lemma \ref{lem:insert2}, if we start with a (rather) inessential map $f: X\to X$, then the resulting map $F\circ f\circ G: Y\to Y$ is again (rather) inessential. The two endo d-maps $f_2\circ f_1$ and $f_2\circ G\circ F\circ f_1$ on $X$ from above ``differ'' only by the inserted (rather) inessential map $G\circ F: X\to X$. 

\begin{proposition}
Let $F: X\to Y$ denote an $\alpha\mc{F}$ equivalence.
\begin{enumerate}
    \item The composite functor $\xymatrix{E\vec{\pi}_1(X)\ar[r]^{E\vec{\pi}_1(F)} & E\vec{\pi}_1(Y)\ar[r] & \Sigma_{\mc{F}}^{\alpha}E\vec{\pi}_1(Y)[\Sigma_{\mc{F}}^{\alpha}(Y)^{-1}]}$ is essentially onto. Every morphism in the localized category is conjugate (via isomorphisms) to a morphism in the image.
    \item The above functor yields a quotient isomorphism $\vec{\pi}_0(F): \vec{\pi}_0(X; \alpha, \mc{F})_*^*\to\vec{\pi}_0(Y; \alpha,\mc{F})_*^*$ on the level of pair component categories.
\end{enumerate}
\end{proposition}
 
In particular, if two d-spaces have non-isomorphic pair component categories $\vec{\pi}_0(-; \alpha , \mc{F})_*^*$, then there cannot exist an $\alpha\mc{F}$-equivalence between them. As Example \ref{ex:wedge2}(2) shows, the existence of an isomorphism of pair component categories is a necessary but \emph{not} a sufficient condition for the existence of an $\alpha\mc{F}$-equivalence.

\begin{proof}
\begin{enumerate}
\item Let $G: Y\to X$ denote an $\alpha\mc{F}$-inverse to $F$. An object $(y_1, y_2)\in\vec{Y}^2$ is connected -- by the isomorphism $FG$ -- to $(FGy_1, FGy_2)=F(Gy_1,Gy_2)\in\vec{Y}^2$ with $(Gy_1,Gy_2)\in\vec{X^2}$. Every covariant extension morphism $*\sigma : (y_1,y_2) \to (y_1,y_3)$ fits into the diagram
\[\xymatrix{(FGy_1,FGy_2)\ar[r]^{*FG\sigma} & (FGy_1,FGy_3)\\
(y_1,y_2)\ar[u]^{FG}\ar[r]^{*\sigma} & (y_1,y_3)\ar[u]_{FG}.}\]
Note that the top morphism $*FG\sigma=F(*G\sigma:(Gy_1,Gy_2)\to (Gy_1,Gy_3))$ is in the image of $E\vec{\pi}_1(F)$. Similarly for contravariant extensions. 

An inessential d-map $f: Y\to Y$ inducing $f:(y_1,y_2)\to (fy_1,fy_2)$ fits into the diagram
\[\xymatrix{(FGy_1,FGy_2)\ar[r]^{=} & (FGy_1,FGy_2)\\
(y_1,y_2)\ar[u]^{FG}\ar[r]^{f} & (fy_1,fy_2)\ar[u]_{FG\circ f^{-1}}}\]
with a morphism in the image of $F$ (of the identity on $(Gy_1,Gy_2)$) on top.
\item As mentioned before, an $\alpha\mc{F}$-equivalence $F: X\to Y$ does not yet induce a functor between localized categories. But inessential maps give rise to isomorphisms in those and hence to \emph{identities} in the pair component categories; hence $\vec{\pi}_0(F)$ becomes a functor. The remaining part follows from (1) above: The functors induced by $FG$ and by $GF$ on pair component categories are the identities.
\end{enumerate}
\end{proof}

\begin{example}
In his thesis \cite{Dubut:17}, Dubut defined an interesting $\Box$-space $D$ consisting of four 2-dimensional cubes and a glueing that can visualized as in Figure \ref{fig:Dub1}: 

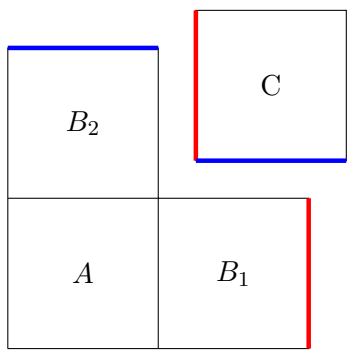
\begin{figure}[h]\label{fig:Dub1}
\begin{tikzpicture}
  \draw (0,0) -- (4,0) -- (4,2) -- (2,2) -- (2,4) -- (0,4) -- (0,0);
\draw (2.5,2.5) -- (4.5,2.5) -- (4.5,4.5) -- (2.5,4.5) -- (2.5,2.5);
\draw (0,2) -- (2,2);
\draw (2,0) -- (2,2);
\node at (1,1) {$A$};
\node at (1,3) {$B_2$};
\node at (3,1) {$B_1$};
\node at (3.5,3.5) {C};
\draw[line width=0.6mm, color=red] (4,0) -- (4,2);
\draw[line width=0.6mm, color=blue] (0,4) -- (2,4);
\draw[line width=0.6mm, color=red] (2.5,2.5) -- (2.5,4.5);
\draw[line width=0.6mm, color=blue] (2.5,2.5) -- (4.5,2.5);
\end{tikzpicture}
\caption{The cubical complex $D$}\label{fig:Dub1}
\end{figure}

It is quite easy to see (cf \cite[Section 1.2]{Raussen:18}) that all arising path spaces in $D$ are either empty, contractible, or they consist of two contractible components. This fact, and also all arising extension maps, are very reminiscent of the space $\partial\Box_2$, the boundary of a square, cf Example \ref{ex:sphere}. But these two spaces are \emph{not} homotopy equivalent: Their component categories are analyzed in \cite{Raussen:18}. The pair component category of $\partial\Box_2$ is easy to understand; it has nine objects; cf \cite[Section 4.3.1]{Raussen:18}. In contrast, the pair component category of $D$ is difficult to describe; it has far more objects; cf \cite[Proposition 5.13]{Raussen:18}. In particular, there cannot exist any directed homotopy equivalence connecting them. 

\end{example}

\end{document}